%% file: main.tex
\def\acmversion{0} 

\ifnum\acmversion=0

\documentclass{article}
\def\draft{1}

\input{header}

\input{macros}
\newcommand{\mnote}[1]{\ifnum\draft=1 {\color{red} [\textbf{MS}: #1]}\fi}
\newcommand{\snote}[1]{\ifnum\draft=1 {\color{orange} [\textbf{NS:} #1]}\fi}

\title{Point-hyperplane incidence geometry and the log-rank conjecture}

\author{Noah Singer\thanks{Harvard College, Harvard University. \href{mailto:noahsinger@college.harvard.edu}{noahsinger@college.harvard.edu}. Supported by the Herchel Smith-Harvard Undergraduate Science Research Program.} \and Madhu Sudan\thanks{School of Engineering and Applied Sciences, Harvard University. \href{mailto:madhu@cs.harvard.edu}{madhu@cs.harvard.edu}. Supported in part by a Simons Investigator Award and NSF Award CCF 1715187.}}

\else 

\documentclass[acmsmall]{acmart}

\AtBeginDocument{%
  \providecommand\BibTeX{{%
    \normalfont B\kern-0.5em{\scshape i\kern-0.25em b}\kern-0.8em\TeX}}}

\setcopyright{acmcopyright}
\copyrightyear{2018}
\acmYear{2018}
\acmDOI{10.1145/1122445.1122456}

\acmJournal{JACM}
\acmVolume{37}
\acmNumber{4}
\acmArticle{111}
\acmMonth{8}

\usepackage[capitalize]{cleveref}
\usepackage{xcolor}

\input{macros}
\fi 
\begin{document}

\ifnum\acmversion=1
\title{Point-hyperplane incidence geometry and the log-rank conjecture}

\author{Noah Singer}
\email{noahsinger@college.harvard.edu}
\orcid{0000-0002-0076-521X}
\affiliation{%
  \institution{Harvard University}
  \city{Cambridge}
  \state{Massachusetts}
  \country{USA}
  \postcode{02138}
}

\author{Madhu Sudan}
\email{madhu@cs.harvard.edu}
\orcid{0000-0003-3718-6489}
\affiliation{%
  \institution{Harvard University}
  \city{Cambridge}
  \state{Massachusetts}
  \country{USA}
  \postcode{02138}
}


\else 

\maketitle 

\fi 

\begin{abstract}
We study the log-rank conjecture from the perspective of point-hyperplane incidence geometry. We formulate the following conjecture: Given a point set in $\BR^d$ that is covered by constant-sized sets of parallel hyperplanes, there exists an affine subspace that accounts for a large (i.e., $2^{-{\polylog(d)}}$) fraction of the incidences, in the sense of containing a large fraction of the points and being contained in a large fraction of the hyperplanes. In other words, the point-hyperplane incidence graph for such configurations has a large complete bipartite subgraph. Alternatively, our conjecture may be interpreted linear-algebraically as follows: Any rank-$d$ matrix containing at most $O(1)$ distinct entries in each column contains a submatrix of fractional size $2^{-{\polylog(d)}}$, in which each column is constant. We prove that our conjecture is equivalent to the log-rank conjecture; the crucial ingredient of this proof is a reduction from bounds for parallel $k$-partitions to bounds for parallel $(k-1)$-partitions. We also introduce an (apparent) strengthening of the conjecture, which relaxes the requirements that the sets of hyperplanes be parallel.

Motivated by the connections above, we revisit well-studied questions in point-hyperplane incidence geometry \emph{without} structural assumptions (i.e., the existence of partitions). We give an elementary argument for the existence of complete bipartite subgraphs of density $\Omega(\epsilon^{2d}/d)$ in any $d$-dimensional configuration with incidence density $\epsilon$, qualitatively matching previous results proved using sophisticated geometric techniques. We also improve an upper-bound construction of Apfelbaum and Sharir~\cite{AS07}, yielding a configuration whose complete bipartite subgraphs are exponentially small and whose incidence density is $\Omega(1/\sqrt d)$. Finally, we discuss various constructions (due to others) of products of Boolean matrices which yield configurations with incidence density $\Omega(1)$ and complete bipartite subgraph density $2^{-\Omega(\sqrt d)}$, and pose several questions for this special case in the alternative language of extremal set combinatorics.

Our framework and results may help shed light on the difficulty of improving Lovett's $\tilde{O}(\sqrt{\rank(f)})$ bound~\cite{Lov16} for the log-rank conjecture. In particular, any improvement on this bound would imply the first complete bipartite subgraph size bounds for parallel $3$-partitioned configurations which beat our generic bounds for unstructured configurations.
\end{abstract}

\ifnum\acmversion=1

\ccsdesc[500]{Theory of computation~Randomness, geometry and discrete structures}
\ccsdesc[500]{Theory of computation~Communication complexity}
\ccsdesc[300]{Mathematics of computing~Combinatorics}
\ccsdesc[500]{Mathematics of computing~Graph theory}

\keywords{log-rank conjecture, incidence geometry, low-rank matrices, extremal set theory}
\fi 

\maketitle

\section{Introduction}

In this work we present several linear-algebraic, incidence-geometric, and set-theoretic conjectures which are connected to the ``log-rank conjecture'' in communication complexity. We also describe some mild progress on the incidence-geometric questions. We start with some background on communication complexity and incidence geometry.

\subsection{Motivation and background}

\subsubsection*{Notation}

$f \leq \tilde{O}(g)$ denotes ``$f \leq O(g \cdot \polylog(g))$.'' $[n]$ denotes the set of integers $\{1,\ldots,n\}$. All logarithms are base 2. A \emph{submatrix of} or \emph{rectangle in} a matrix $M \in \BR^{\CX \times \CY}$ is given by two subsets $\CA \subseteq \CX$ and $\CB \subseteq \CY$ and denoted $M|_{\CA \times \CB}$; we use these terms interchangeably. In real space $\BR^d$, a \emph{$j$-flat} is a $j$-dimensional affine subspace. Hence, a point is a 0-flat and a hyperplane is a $(d-1)$-flat. $2^{\CX}$ denotes the powerset of a set $\CX$.

\subsubsection{Communication complexity and the log-rank conjecture}\label{sec:motivation}
The \emph{(deterministic) communication complexity} of a (two-party) function $f : \CX \times \CY \to \{0,1\}$, as defined by Yao~\cite{Yao79}, measures how much communication is needed for two cooperating parties, one knowing $x \in \CX$ and the other knowing $y \in \CY$, to jointly determine $f(x,y)$. The (deterministic) communication complexity $\CC_{det}(f)$ is the minimum over all communication protocols that compute $f(x,y)$ of the maximum communication over all pairs $(x,y)\in \CX\times\CY$. 

Every function $f : \CX \times \CY \to \{0,1\}$ corresponds naturally to a Boolean matrix $M_f$, with rows indexed by $\CX$ and columns by $\CY$, where $(M_f)_{x,y}=f(x,y)$. Given a function $f : \CX \times \CY \to \{0,1\}$, we can define its \emph{rank}, denoted $\rank(f)$, as the rank of $M_f$  over $\BR$, which is a linear-algebraic measure of $f$'s complexity. This leads to a natural question: How is $\rank(f)$ connected to $\CC_{det}(f)$?

A \emph{monochromatic rectangle} for a function $f : \CX \times \CY \to \{0,1\}$ is a pair $\CA \subseteq \CX, \CB \subseteq \CY$ such that $f(a,b)$ is constant over all $(a,b) \in \CA \times \CB$. A $c$-bit communication protocol for $f$ partitions the space $\CX \times \CY$ into a disjoint union of at most $2^c$ monochromatic rectangles. Since monochromatic rectangles for $f$ correspond to rank-1 submatrices of $M_f$, $\log(\rank(f)) \leq \CC_{det}(f)$~\cite{MS82}. The \emph{log-rank conjecture} of Lov\'asz and Saks~\cite{LS88} posits a matching upper bound up to a polynomial factor; that is:

\begin{conjecture}[Log-rank conjecture~\cite{LS88}]\label{conj:log-rank}
	For every function $f : \CX \times \CY \to \{0,1\}$, \[ \CC_{det}(f) \leq \polylog(\rank(f)). \]
\end{conjecture}

This conjecture is a central and notorious open question in communication complexity. Currently, the best known bound for arbitrary $f$ is due to Lovett~\cite{Lov16}, who proved the following:

\begin{theorem}[{\cite{Lov16}}]\label{thm:lovett-sqrt-rank}
For every function $f : \CX \times \CY \to \{0,1\}$, \[ \CC_{det}(f) \leq O\left(\sqrt{\rank(f)} \log(\rank(f))\right). \]
\end{theorem}

The log-rank conjecture asserts that every low-rank Boolean matrix can be partitioned into a small number of monochromatic rectangles. An obviously necessary condition for this is the presence of a large monochromatic rectangle. A result
due to Nisan and Wigderson~\cite{NW95} shows that this is in fact also a sufficient condition. Specifically, define the \emph{size} of a rectangle $(\CA,\CB)$ as $|\CA||\CB|$. Then:

\begin{theorem}[\cite{NW95}, as articulated in~\cite{Lov16}]\label{thm:NW95}
	Suppose that there exists some function $\gamma : \BN \to \BN$ such that the following is true: For every Boolean function $f : \CX \times \CY \to \{0,1\}$, $f$ contains a monochromatic rectangle of size at least $|\CX||\CY|\cdot2^{-\gamma(\rank(f))}$. Then for every Boolean function $f : \CX \times \CY \to \{0,1\}$, \[ \CC_{det}(f) \leq O\left(\log^2(\rank(f)) + \sum_{i=0}^{\log(\rank(f))} \gamma\left(\frac{\rank(f)}{2^i}\right)\right). \]
\end{theorem}

In particular, proving that the hypothesis of this theorem holds with $\gamma(d) = \polylog(d)$ would suffice to prove the log-rank conjecture.\footnote{This reduction is tight in a strong sense: A $c$-bit protocol for $f$ partitions $\CX \times \CY$ into $\leq 2^c$ monochromatic rectangles, one of which must have size at least $|\CX||\CY|\cdot2^{-c}$.} See \cite{Lov20} for more background on the log-rank conjecture.

\subsubsection{Incidence geometry and extremal combinatorics}\label{sec:inc-geo-background}

In this section, we give various definitions and notations that we will use throughout the rest of the paper.

In $\BR^d$, a \emph{hyperplane} is the locus of points $x \in \BR^d$ defined by an equation of the form $\langle a, x \rangle = b$, for some $a \neq 0 \in \BR^d, b \in \BR$. We refer to the vector $a$ as the \emph{normal vector} of $h$ and $b$ as its \emph{offset} (denoted $b(h)$). A pair of hyperplanes $h,h'$ are \emph{parallel} if for some constant $c \in \BR$ we have $a = ca'$ (where $h$ and $h'$ are defined by $\langle a,x \rangle = b$ and $\langle a',x \rangle = b'$, respectively).

A point $p$ and a hyperplane $h$ in $\BR^d$ are \emph{incident} if $p$ lies on $h$; we call the pair $(p,h)$ an \emph{incidence} (and say $p$ is \emph{incident to} $h$ and vice versa). In an ambient space $\BR^d$, we refer to a (finite) set $\CP$ of points together with a (finite) set $\CH$ of hyperplanes as a \emph{configuration}.  Configurations determine an \emph{incidence graph} $\Gr(\CP,\CH)$, an (unweighted, undirected) bipartite graph defined as follows: The left vertices are the points $\CP$, the right vertices are the hyperplanes $\CH$, and the edge $(p,h)$ is included iff $p$ is incident to $h$. Following Apfelbaum and Sharir~\cite{AS07}, we denote by $\I(\CP,\CH)$ the total number of incidences between $\CP$ and $\CH$ (equiv., the number of edges in $\Gr(\CP,\CH)$) and by $\rs(\CP,\CH)$ the largest number of edges in any complete bipartite subgraph of $\Gr(\CP,\CH)$; the reader may verify the equivalent characterization that \[ \rs(\CP,\CH) = \max_{S \text{ affine subspace } \subset \BR^d} (|\{p \in \CP: p \text{ lies on } S\}| \cdot |\{h \in \CH: S \text{ lies on } h\}|). \] Let $|\CP| = n$ and $|\CH| = m$; we refer to the ratios $\frac{\I(\CP,\CH)}{mn}$ and $\frac{\rs(\CP,\CH)}{mn}$ as the \emph{incidence} and \emph{complete bipartite subgraph densities} of the configuration $(\CP,\CH)$, respectively.\footnote{A quick note on our use of Apfelbaum and Sharir~\cite{AS07}'s notation: We use $n$ to denote the number of points and $m$ to denote the number of hyperplanes, which is opposite to \cite{AS07}. Also, \cite{AS07} uses $\Pi$ instead of $\CH$ to denote the set of hyperplanes. Finally, for context, in \cite{AS07}'s notation $\rs(\CP,\CH)$, $r$ refers to the quantity $|\{p \in \CP: p \text{ lies on } S\}|$ and $s$ refers to the quantity $|\{h \in \CH: S \text{ lies on } h\}|$; $\rs(\CP,\CH)$ maximizes the product $rs$ over all affine subspaces of $\BR^d$.}

Generally, bipartite graphs need not contain large complete bipartite subgraphs; indeed, even in random bipartite graphs of constant edge density, the largest complete bipartite subgraphs are logarithmically small \cite{Erd47}, which is tight up to constant factors \cite{ES35}. But lower bounds for complete bipartite subgraph density (and various analogues in hypergraphs) have been widely studied for specific classes of (hyper)graphs with ``structure''. Previous research has explored complete bipartite subgraph density and its analogues in settings including point-hyperplane incidences \cite{AS07,Do20}, line segment incidences \cite{PS01}, orientations of $k$-tuples of points \cite{BV98}, and common points in simplices defined by $k$-tuples of points \cite{Pac98,KKP+15} or bounded by $k$-tuples of hyperplanes \cite{BP14}. All of these problems fit into the framework of ``semi-algebraic'' hypergraphs, which arise from solutions to systems of ``low-complexity'' polynomial equations; complete bipartite sub-hypergraphs in semi-algebraic hypergraphs have been studied in \cite{APP+05,FGL+12,CFP+14,FPS16}. Non-geometric settings have also been studied, including graphs excluding a fixed induced subgraph \cite{EHP00} and graphs with small VC-dimension \cite{FPS19}. 

To the best of our knowledge, prior to the current work, the strongest known bounds for complete bipartite subgraph density in point-hyperplane incidence graphs were as follows. Fox, Pach, and Suk~\cite{FPS16} proved the following upper bound (which holds more generally for semi-algebraic relations) using a cell decomposition argument:

\begin{theorem}[{\cite[Corollary 1.2]{FPS16}} with $k = t = 2$]\label{thm:fps-lower-bound}
	Let $\mathcal{P}$ and $\mathcal{H}$ be a set of $n$ points and $m$ hyperplanes, respectively, in $\mathbb{R}^d$, and let $\epsilon = \frac{\I(\mathcal{P},\mathcal{H})}{mn}$. Then \[ \rs(\mathcal{P}, \mathcal{H}) \geq \Omega(\epsilon^{d+1} 2^{-d(40\log(d+1)+1)}\cdot mn). \]
\end{theorem}

Apfelbaum and Sharir~\cite[Theorem 1.2]{AS07} proved the related bound $\rs(\CP,\CH) \geq \Omega_d(\epsilon^{d-1} \cdot mn)$, but did not analyze the dependence on the dimension $d$; we suspect it is worse than the $2^{-\tilde{O}(d)}$ in \cref{thm:fps-lower-bound}. \cite{AS07} also proved the currently-best upper bound, using a lattice construction inspired by Elekes and T\'oth~\cite{ET05}:

\begin{theorem}[{\cite[Theorem 1.3]{AS07}}]\label{thm:as-upper-bound}
	For every $d \in \BN$, there exist arbitrarily large $n,m \in \BN$ such that there exists a set $\CP$ of $n$ points and a set $\CH$ of $m$ hyperplanes in $\BR^d$ such that $\I(\CP,\CH) \geq \Omega(nm/d)$ and $\rs(\CP,\CH) \leq O(mn \cdot 2^{-d} / \sqrt d)$.
\end{theorem}

\subsubsection{Connecting low-rank Boolean matrices and incidence geometry}\label{sec:connecting-logrank-incgeo}

The rank $d$ of a Boolean matrix $M$ has a natural incidence-geometric interpretation, as observed in unpublished work by Golovnev, Meka, Sudan, and Velusamy~\cite{GMSV19}. Given a factorization $M=PH$ where $P\in\BR^{n \times d},H\in\BR^{d \times m}$, we view the $i$-th row of $P$ as a point in $\BR^d$, and the $j$-th column of $H$ as the normal vector of a pair of parallel hyperplanes with offsets 0 and 1. Then the entry at a particular row and column in $M$ determines which of the corresponding hyperplanes the corresponding point is incident to. The hope that there is a large monochromatic rectangle in $M$ translates to the hope that there is a large complete bipartite subgraph in the point-hyperplane incidence graph, or equivalently that there is an affine subspace that is contained in many hyperplanes and contains many points. Moreover, the fact that the matrices under consideration are Boolean matrices implies that these point-hyperplane configurations have an unusually large density of incidences (specifically 50\% of the point-hyperplane pairs are incident!).

One could ask if simply the high density of incidences suffices to imply the existence of a large complete bipartite subgraph (of density $2^{-{\polylog(d)}}$ in $d$ dimensions). This is known to be false and a 2016 construction of Lovett~\cite{Lov16} with density $2^{-\Theta(\sqrt d)}$ is a counterexample.\footnote{Though admittedly the authors were not aware of this at earlier stages of this writing~\cite{SS21-early-version}. We thank Lovett~\cite{Lov16}, \Palvolgyi~\cite{Pal21}, and Fox and Wigderson~\cite{FW21} for pointing this out to us, and for suggesting related counterexamples.} See \cref{sec:set-system-upper-bound} for discussion on this and related constructions.

\subsection{Contributions}

\subsubsection{Moderately large complete bipartite subgraphs in general configurations}\label{sec:general-conj}

As discussed in the previous subsection, there exist configurations with incidence density $\Omega(1)$ but complete bipartite subgraph density $2^{-\Omega(\sqrt d)}$. Only assuming incidence density $\Omega(1)$, the following conjecture is the strongest possible bound that may yet turn out to be true; it is roughly an incidence-geometric analogue of a conjecture of Lovett on sparse low-rank matrices~\cite[Conjecture 5.1]{Lov16}:

\begin{conjecture}[Configurations with incidence density $\Omega(1)$ have moderately large complete bipartite subgraphs]\label{conj:general-conj}
The following is true for every fixed $\epsilon > 0$. In $\BR^d$, let $\CP$ be a collection of $n$ points and $\CH$ a collection of $m$ hyperplanes such that $\I(\CP,\CH) \geq \epsilon \cdot mn$. Then \[ \rs(\CP,\CH) \geq mn \cdot 2^{-{O_{\epsilon}(\sqrt d)}}. \] (The constant in $O_\epsilon(\sqrt d)$ may depend arbitrarily on $\epsilon$.)
\end{conjecture}

We note that \cref{conj:general-conj} is too weak to prove the log-rank conjecture (\cref{conj:log-rank}). It would, however, yield a result roughly matching the current best upper bound (\cref{thm:lovett-sqrt-rank}, due to \cite{Lov16}) on communication complexity as a function of the rank (up to some logarithmic factors). 

\subsubsection{Large complete bipartite subgraphs in ``structured'' configurations}

In view of the issues discussed in \cref{sec:connecting-logrank-incgeo,sec:general-conj}, we consider what \emph{additional} properties of point-hyperplane configurations could potentially lead to the presence of a large complete bipartite subgraph. As discussed in \cref{sec:connecting-logrank-incgeo}, Boolean matrices lead to configurations where the hyperplanes can be partitioned into {\em parallel pairs} such that each pair {\em covers} the set of points. This leads to an easy reformulation of the log-rank conjecture (\cref{conj:original-conj} below). We extend this formulation to the notion of \emph{parallel $k$-partitions} (see \cref{defn:parallel-partition} below) that allow the sets of hyperplanes to be partitioned into sets of size (at most) $k$ for an arbitrary constant $k$. We conjecture that $k$-parallel partitionable configurations contain large complete bipartite subgraphs (\cref{conj:constant-sized-parallel-partitions} below), and in \cref{thm:k-reduction} below, we show that this conjecture is actually equivalent to the log-rank conjecture. \autoref{thm:k-reduction} has a natural interpretation as a \emph{reduction} from bounds for parallel $k$-partitioned configurations to (stronger) bounds for $(k-1)$-partitioned configurations.

In the matrix corresponding to a point-hyperplane configuration, there is a natural linear-algebraic property equivalent to the presence of a parallel $k$-partition, which we refer to as \emph{$k$-listability}: A matrix is $k$-listable if every column has at most $k$ distinct entries (though these sets may differ arbitrarily across columns). \cref{thm:k-reduction} can hence be stated linear-algebraically, in closer spirit to the original log-rank conjecture: the log-rank conjecture is equivalent to the assertion that every $k$-listable matrix contains a large $1$-listable submatrix. (See \cref{sec:conjectures} for a more careful account of the connection between matrices, configurations, listability, and parallel partitionability.)

We believe that the {\em covering} aspect of the point-hyperplane incidences is a key element of the log-rank conjecture and posit an extension (which does not immediately seem to be equivalent to the log-rank conjecture, nor does it seem to have a simple linear-algebraic formulation). Specifically in \cref{conj:constant-sized-partitions} we suggest that if a set of hyperplanes can be partitioned (in a not-necessarily-parallel way) into blocks of size at most $k$ such that each block covers a given set of points, then the incidence graph corresponding to this configuration has a large complete bipartite subgraph.

\subsubsection{New bounds for complete bipartite subgraph size in general configurations}

Returning to the more basic question of the incidence density of a configuration versus the size of its largest complete bipartite subgraph, we present two results that improve the state of the art. Our first result here is a lower bound on the size of complete bipartite subgraphs in incidence graphs of constant density:

\begin{theorem}\label{thm:probabilistic-lower-bound}
	Let $\mathcal{P}$ and $\mathcal{H}$ be a set of $n$ points and $m$ hyperplanes, respectively, in $\mathbb{R}^d$, such that $\I(\mathcal{P},\mathcal{H}) \geq \epsilon mn$. If $n$ is sufficiently large (in particular, if $\frac{\epsilon^d}2n > 1$), then \[ \rs(\mathcal{P}, \mathcal{H}) \geq \Omega\left(\frac{\epsilon^{2d}}{d} mn\right). \]
\end{theorem}

\cref{thm:probabilistic-lower-bound} is proven in \cref{sec:probabilistic-lower-bound} using an elementary probabilistic argument. Compared to the previously-best lower bound we are aware of (\cref{thm:fps-lower-bound}, due to Fox, Pach, and Suk~\cite{FPS16}), for fixed $\epsilon$ we get slightly better dependence on $d$ (exponentially small in $O(d)$ instead of $O(d \log d)$), and our proof is also drastically simpler.

On the flip side, we are able to quantitatively improve the previously-best upper bound (\cref{thm:as-upper-bound}, due to Apfelbaum and Sharir~\cite{AS07}):

\begin{theorem}\label{thm:lattice-upper-bound}
	For every $d > 0$, there exists a set $\CP$ of $n$ points and a set $\CH$ of $m$ hyperplanes in $\BR^d$ such that $\I(\CP,\CH) \geq \Omega(nm/\sqrt d)$ and $\rs(\CP,\CH) \leq O(mn \cdot 2^{-d} / \sqrt d)$.
\end{theorem}

The only difference between \cref{thm:lattice-upper-bound} and the earlier \cref{thm:as-upper-bound} is the gain of a $\sqrt d$ factor in the lower bound on $\I(\CP,\CH)$, which we achieve by exhibiting a dense subset of \cite{AS07}'s construction (see \cref{claim:dense-subgraph} below). (That is, \cref{thm:as-upper-bound} has only $\I(\CP,\CH) = \Theta(mn /d)$. Other explicit constructions, discussed in \cref{sec:set-system-upper-bound}, have $\I(\CP,\CH) \geq \Omega(mn)$ but only $\rs(\CP,\CH) \leq O(mn \cdot 2^{-\Omega(\sqrt d)})$~\cite{Lov16,Lov21,Pal21,FW21}.) But since the incidence density of \cref{thm:lattice-upper-bound}'s construction is still not constant, it is too small to falsify our most general conjecture (\cref{conj:general-conj}).

\subsubsection{Explicit constructions of configurations without large complete bipartite subgraphs, and products of Boolean matrices}

In \cref{sec:set-system-upper-bound}, we analyze an explicit point-hyperplane configuration which was suggested as a counterexample by \Palvolgyi~\cite{Pal21}. Our analysis demonstrates that $O(\sqrt d)$ is the best possible exponent in \cref{conj:general-conj}. (A different counterexample with the same quantitative parameters and a related analysis, due to Lovett~\cite{Lov16} and also indicated by Fox and Wigderson~\cite{FW21}, was already known.) More generally, we observe that these constructions all arise from products of Boolean matrices, and we pose special cases of our general conjectures for these types of matrices as interesting variants. Since they involve products of Boolean matrices, we also describe them in the language of extremal combinatorics of set families and connect them to prior results in those areas.

\subsubsection{Discussion}

We believe that our new incidence-geometric and linear-algebraic frameworks for interpreting the log-rank conjecture shed light on several important prior results: in particular, the positive results of Lovett~\cite{Lov16} (i.e., \cref{thm:lovett-sqrt-rank}), which in our incidence-geometric language is a complete bipartite subgraph density lower bound of $2^{-\tilde{O}(\sqrt d)}$ for parallel 2-partitioned configurations, and the constructions of \cite{Lov16,Lov21,Pal21,FW21} of (unstructured) configurations with incidence density $\Omega(1)$ and complete bipartite subgraph density $2^{-\Omega(\sqrt d)}$. For context, \cite{Lov16}'s analysis relies heavily on the \emph{binarity} of the matrix; roughly, monochromatic rectangles are created by using a hyperplane rounding argument that exploits the gap between the two possible values for entries in the matrix. Altogether we are left in the following unsettling situation: The only way we know to prove a $2^{-O(\sqrt d)}$ bound uses binarity, but we believe that we should be able to get (1) a $2^{-{\polylog(d)}}$ bound using binarity (i.e., \cref{conj:original-conj}) \emph{and} (2) a $2^{-O(\sqrt d)}$ bound without using structural assumptions (i.e., \cref{conj:general-conj}).

This situation is especially interesting in light of the reduction used to prove \cref{thm:k-reduction}. According to this reduction, if we have a $2^{-f(d)}$ complete bipartite subgraph density lower bound for parallel $(k-1)$-partitioned configurations, then we also have a $2^{-(f(d))^2}$ bound for parallel $k$-configurations. In particular, since the best bound we know for parallel $2$-partitioned configurations is $2^{-\tilde{O}(\sqrt d)}$ (i.e., \cref{thm:lovett-sqrt-rank}), the best bound we know for parallel $3$-partitioned configurations is $2^{-\tilde{O}(d)}$ --- which simply recovers what we already proved in \cref{thm:probabilistic-lower-bound}! Indeed, even a modest improvement in the bounds for parallel $2$-partitioned configurations would yield a nontrivial bound for parallel $3$-partitioned configurations. While this is potentially due to a technical weakness of the reduction, it may still help explain the difficulty in surpassing the ``$\sqrt d$ barrier''.

\section{Incidence-geometric reformulations of the log-rank conjecture}\label{sec:conjectures}

In this section we present some reformulations of the log-rank conjecture in terms of incidence geometric questions. The conjectures start with some unpublished work of Golovnev, Meka, Sudan and Velusamy~\cite{GMSV19} who raised \cref{conj:original-conj} below explicitly and also went on to propose a stronger form of \cref{conj:general-conj} (which was~\cite[Conjecture 5]{SS21-early-version}). The latter turns out to be false (and this was already known --- see~\cite{Lov16} and \cref{sec:set-system-upper-bound}), so we propose several new variants here and prove some equivalences.

\subsection{The ``original'' reformulation}\label{sec:original-relax}

To begin, we introduce new notions of structured point-hyperplane configurations.

\begin{definition}[Parallel $k$-partition]\label{defn:parallel-partition}
Let $(\CP,\CH)$ be a point-hyperplane configuration. A \emph{parallel $k$-partition} for $(\CP,\CH)$ is a partition of $\CH$ into disjoint blocks $\CH_1 \sqcup \cdots \sqcup \CH_\ell$, each of size at most $k$, such that (1) within each block $\CH_i$, the hyperplanes all have the same normal vector, and (2) for each block $\CH_i$ and point $p \in \CP$, $p$ is incident to one of the hyperplanes of $\CH_i$.
\end{definition}

Note that in a parallel $k$-partitioned configuration, every point is incident to \emph{precisely} one hyperplane in each block.

Next, we define properties of matrices which we will soon show are analogous to parallel $k$-partitionability:

\begin{definition}[$k$-listability and $k$-arity]
A matrix $M\in\BR^{n \times m}$ is \emph{$k$-listable} if every column of $M$ contains at most $k$ distinct entries. Moreover, $M$ is \emph{$k$-ary} if it contains at most $k$ \emph{total} distinct entries.
\end{definition}

Note that $1$-listability is equivalent to every column being constant, and that $k$-arity implies $k$-listability.

Next, we describe a natural correspondence between parallel $k$-partitioned configurations and $k$-listable matrices. Specifically, we define a matrix associated with every configuration, and conversely, a configuration associated with every matrix.

Given a configuration $(\CP,\CH)$ of $n$ points and $m$ hyperplanes in $\BR^d$ with a parallel $k$-partition $\CH_1 \sqcup \cdots \sqcup \CH_\ell$, let $\Mat(\CP,\CH)$ be the $n \times m$ matrix defined as follows: Let $p_i$ be the $i$-th point in $\CP$ and let $a_j$ be the normal vector corresponding to the hyperplanes in block $\CH_j$. Then the $(i,j)$-th entry of $\Mat(\CP,\CH)$ is $\langle p_i,a_j \rangle$.

Now we describe how to define a configuration $\Con(M)$ associated with a matrix $M \in \BR^{n \times m}$ of rank $d$. Consider a factorization $M = PQ$ where
$P \in \BR^{n \times d}$ and $Q \in \BR^{d \times m}$. ($\Con(M)$ may depend on the choice of this factorization, but picking an arbitrary one suffices for our purposes.) Let $p_1,\ldots,p_n \in \BR^d$ denote the rows of $P$ and let $q_1,\ldots,q_m \in \BR^d$ denote the columns of $Q$.  For $j \in [m]$ let  $B_j$ denote the set of distinct entries in column $j$ of $M$. For each $j \in [m]$ and $b \in B_j$, define $h_j^b$ as the hyperplane determined by the equation $\langle x, q_j\rangle = b$ over $x \in \BR^d$. We define the configuration \[ \Con(M) := (\{p_i : i \in [n]\}, \{ h_j^b : j \in [m], b \in B_j \}). \]

The basic facts about the correspondence between matrices and configurations are summarized in \cref{prop:mat-con-corr} below. Given a set $\CH$ of hyperplanes, define the \emph{offset set} $\CB(\CH) \subseteq \BR$ as $\CB(\CH) := \{b(h) : h \in \CH\}$.

\begin{proposition}\label{prop:mat-con-corr}
    \begin{enumerate}
        \item If $(\CP,\CH)$ is a parallel $k$-partitioned configuration in $\BR^d$, then $\Mat(\CP,\CH)$ is $k$-listable and $|\CB(\CH)|$-ary and has rank $\leq d$.
        \item If a matrix $M \in \BR^{n \times m}$ is $k$-listable and $\ell$-ary, then $\Con(M)$ has a parallel $k$-partition, $|\CB(\Con(M))| \leq \ell$, and $\Con(M)$ contains between $m$ and $mk$ hyperplanes.
        \item Every $\ell$-ary matrix $M \in \BR^{n \times m}$ has a $1$-listable submatrix of size at least $\rs(\Con(\CP,\CH))$ (and hence a monochromatic rectangle of size at least $\rs(\CP,\CH)/\ell$).
        \item If $(\CP,\CH)$ is a parallel $k$-partitioned configuration, and $\Mat(\CP,\CH)$ contains a monochromatic rectangle of size $t$, then $\rs(\CP,\CH) \geq t$.
    \end{enumerate}
\end{proposition}

The proofs follow immediately from the definitions and so we omit them.

We can use \cref{prop:mat-con-corr} to show that the following conjecture is equivalent to \cref{conj:log-rank}:

\begin{conjecture}[Parallel $2$-partitioned configurations have large complete bipartite subgraphs~\cite{GMSV19}]\label{conj:original-conj}
In $\BR^d$, let $(\CP,\CH)$ be a parallel $2$-partitioned configuration with $\CB(\CH) = \{0,1\}$. Then \[ \rs(\CP,\CH) \geq mn \cdot 2^{-{\polylog(d)}}. \]
\end{conjecture}

\begin{theorem}
The log-rank conjecture (\cref{conj:log-rank}) holds if and only if \cref{conj:original-conj} does.
\end{theorem}

\begin{proof}
($\Longrightarrow$) Given any parallel $2$-partitioned configuration $(\CP,\CH)$ in $\BR^d$, we may assemble the matrix $\Mat(\CP,\CH)$. By \cref{prop:mat-con-corr}, $\rank(\Mat(\CP,\CH)) \leq d$ and $\Mat(\CP,\CH)$ is binary, i.e., $\Mat(\CP,\CH)$ has two distinct entries, $a$ and $b$. Letting $\tilde M := (\Mat(\CP,\CH)-a)/(b-a)$, we have $\rank(\tilde M) \leq \rank(\Mat(\CP,\CH)) +1$. $\tilde M$ is Boolean, so assuming the log-rank conjecture, it contains a monochromatic rectangle of size at least $|\CP| |\CH| \cdot 2^{-{\polylog(d)}}$ (see the discussion above \cref{thm:NW95}). Hence so does $\Mat(\CP,\CH)$, so by \cref{prop:mat-con-corr}, $\rs(\CP,\CH) \geq |\CP| |\CH| \cdot 2^{-{\polylog(d)}}$. ($\Longleftarrow$) Given any Boolean matrix $M \in \{0,1\}^{n \times m}$ of rank at most $d$, we may form the configuration $\Con(\CP,\CH)$ in $\BR^d$, which has by \cref{prop:mat-con-corr} a parallel 2-partition and at least $m$ hyperplanes. Assuming \cref{conj:original-conj}, $\rs(\Con(\CP,\CH)) \geq mn \cdot 2^{-{\polylog(d)}}$. Hence by \cref{prop:mat-con-corr} again, $M$ contains a monochromatic rectangle of size at least $mn \cdot 2^{-{\polylog(d)}-1}$, which suffices by \cref{thm:NW95} to prove the log-rank conjecture.
\end{proof}

\subsection{Relaxations of {\cref{conj:original-conj}} }

We could hope to relax the hypothesis of \cref{conj:original-conj} to only require a parallel partition of constant size:

\begin{conjecture}[Parallel partitioned configurations have large complete bipartite subgraphs]\label{conj:constant-sized-parallel-partitions}
The following is true for every fixed integer $k > 1$. In $\BR^d$, let $(\CP,\CH)$ be a configuration with a parallel $k$-partition. Then \[ \rs(\CP,\CH) \geq mn \cdot 2^{-{\polylog(d)}}. \] Equivalently, by \cref{prop:mat-con-corr}, every $k$-listable matrix $M \in \BR^{n \times m}$ contains a $1$-listable submatrix of size at least $mn \cdot 2^{-{\polylog(\rank(M))}}$.
\end{conjecture}

In \cref{thm:k-reduction} below, we prove that \cref{conj:constant-sized-parallel-partitions} is actually equivalent to \cref{conj:original-conj} (and thus to the log-rank conjecture). We could also relax the \emph{parallel} requirement of partition:

\begin{conjecture}[Partitioned configurations have large complete bipartite subgraphs]\label{conj:constant-sized-partitions}
The following is true for every fixed integer $k > 1$. In $\BR^d$, let $(\CP,\CH)$ be a configuration with a (not-necessarily-parallel) $k$-partition, i.e., such that $\CH$ can be partitioned into \emph{blocks} $\CH_1 \sqcup \cdots \sqcup \CH_\ell$ of size at most $k$ such that for each block $\CH_i$ and point $p \in \CP$, $p$ is incident to \emph{at least} one of the hyperplanes of $\CH_i$. Then \[ \rs(\CP,\CH) \geq mn \cdot 2^{-{\polylog(d)}}. \]
\end{conjecture}

We are currently unable to show that \cref{conj:constant-sized-partitions} is implied by \cref{conj:constant-sized-parallel-partitions}.

\subsection{Equivalence of \cref{conj:original-conj} and \cref{conj:constant-sized-parallel-partitions}}\label{sec:relaxation-proofs}

In this section, we show that \cref{conj:constant-sized-parallel-partitions} is implied by \cref{conj:original-conj}.

\begin{lemma}[Folklore]\label{lemma:poly-rank}
If $M \in \BR^{n \times m}$ has rank $d$, and $p\in \BR[X]$ is a real polynomial, then the matrix $N$ given by $N_{ij} = p(M_{ij})$ for every $(i,j) \in [n] \times [m]$ has rank at most $\sum_{c \in S(p)} d^c$, where $S(p) := \{c \geq 0 : p \text{ contains a nonzero monomial of degree }c\}$.
\end{lemma}

\begin{proof}
Recall that rank is subadditive: If $A$ and $B$ are matrices, then $\rank(A+B) \leq \rank(A) + \rank(B)$. Hence it suffices to show that for every $c$, the matrix $N$ given by $N_{ij} = M_{ij}^c$ has rank at most $d^c$.

If $M$ has rank $d$, we can write $M = PQ$ for some $P \in \BR^{n \times d}, Q \in \BR^{d \times m}$; let $p_i$ and $q_j$ denote the $i$-th row of $P$ and the $j$-th column of $Q$, respectively. We have $M_{ij} = \langle p_i,q_j \rangle$ by definition. Then let $p'_i := p_i^{\otimes c}$, i.e., the $c$-fold self-Kronecker product of $p_i$, which is the $d^c$-dimensional vector whose entries correspond to products of each possible sequence of $c$ elements of $p_i$. Similarly, let $q'_j := q_j^{\otimes c}$. Hence we have \[ N_{ij} = M_{ij}^c = \langle p_i,q_j \rangle^c = \left(\sum_{k=1}^d p_{i,k} q_{j,k}\right)^c = \sum_{k_1,\ldots,k_c \in [d]} p_{i,k_1} q_{j,k_1} \cdots p_{i,k_c} q_{j,k_c} = \langle p'_i, q'_j \rangle, \] where $p_{i,k}$ and $q_{j,k}$ denote the $k$-th entries of $p_i$ and $q_j$, respectively. Hence letting $P' \in \BR^{n \times d^c}$ be the matrix whose $i$-th row is $p'_i$ and $Q' \in \BR^{d^c \times m}$ be the matrix whose $j$-th column is $q'_j$, we have $N = P'Q'$, so $N$ has rank at most $d^c$.
\end{proof}

\begin{theorem}\label{thm:k-reduction}
If the log-rank conjecture holds (in the form of \cref{conj:original-conj}), then \cref{conj:constant-sized-parallel-partitions} holds. In particular, assuming \cref{conj:original-conj}, for all integers $k > 2$, there exists a polynomial $p_k$ such that the following is true: Every $k$-listable, rank-$d$ matrix $M \in \BR^{n \times m}$ has a 1-listable submatrix of size at least $mn \cdot 2^{- p_k(\log d)}$.
\end{theorem}

\begin{proof}
To begin, we argue that it suffices to prove the theorem only for matrices $M$ which (1) have no $1$-listable (i.e., constant) columns and (2) contain a 0 and 1 in every column. Firstly, we reduce to the case where (1) holds: If at least half of $M$'s columns are $1$-listable, then we immediately have a $1$-listable submatrix of $M$ containing all the rows and at least half the columns. Otherwise, we may throw out all the $1$-listable columns, thereby reducing the total number of columns by at most half without increasing the rank. Next, we reduce to the case where (2) holds as well. Since (1) holds, we can let $\{a_j,b_j\}_{j \in [m]}$ with $a_j \ne b_j$ be such that $j$-th column of $M$ contains $a_j$ and $b_j$. Let $A \in \BR^{n\times m}$ be the rank-1 matrix with column $j$ being the constant vector $(a_j,\ldots,a_j)$. Let $D \in \BR^{m \times m}$ be the diagonal matrix with $(j,j)$-th entry being $1/(b_j - a_j)$. Now let $N := (M - A) D$. Then $\rank(N) \leq \rank(M-A) \leq d+1$ by subadditivity of rank, and moreover every column of $N$ contains a $0$ and a $1$. And proving the theorem for $N$ immediately implies the theorem for $M$, since if $S \subseteq [n], T \subseteq [m]$ are such that $N|_{S \times T}$ is a $1$-listable submatrix of $N$, then $(N+AD)|_{S \times T}$ is $1$-listable and hence so is $M|_{S \times T}$.

Now, we will prove the theorem by induction on $k$. The $k=2$ case is implied by the log-rank conjecture (in the form of \cref{conj:original-conj}), since if every column of $M$ is $2$-listable and contains a 0 and 1, $M$ is precisely a Boolean matrix.

Let $p_2$ be the polynomial given by the log-rank conjecture. For general $k$, assume the theorem holds for $k-1$, and let $p_k(x)$ be a polynomial satisfying $p_k(x) \geq p_{k-1}(\log(4^x+2^x)) + p_2(x)$ for sufficiently large $x$ (e.g., $p_k(x) = p_{k-1}(x^2+1) + p_2(x)$). For an arbitrary $k$-listable, rank-$d$ matrix $M$ with a 0 and a 1 in every column, let $\tilde{M}$ be the matrix with $\tilde{M}_{ij} := M_{ij}(M_{ij}-1)$. $\tilde{M}$ has rank at most $d^2+d$ by \cref{lemma:poly-rank}. Also $\tilde{M}$ is $(k-1)$-listable (since $0$'s and $1$'s in $M$ become $0$'s in $\tilde{M}$). So by induction $\tilde M$ has a submatrix $\tilde{M}|_{S \times T}$ which is $1$-listable, and $|S||T| \geq mn \cdot 2^{p_{k-1}(\log (d^2+d))}$. Now $M|_{S \times T}$ is $2$-listable, since each column of $\tilde M|_{S \times T}$ is some constant value $c$, and any value in the corresponding column of $M|_{S \times T}$ must be a root of $z(z-1)=c$. Moreover, $\rank(M|_{S \times T}) \leq \rank(M) = d$, since the rank of a submatrix never exceeds the original matrix's rank. We thus conclude, now using the base case $k=2$, that there exist $S' \subseteq S$ and $T' \subseteq T$ with $M|_{S' \times T'}$ being $1$-listable and $|S'||T'| \geq |S||T| \cdot 2^{- p_2(\log d)}$. Combining the above we have $|S'||T'| \geq mn \cdot 2^{- p_2(\log d)}\cdot 2^{- p_{k-1}(\log (d^2+d))} \geq mn \cdot 2^{-p_k(\log d)}$ by assumption on $p_k$.
\end{proof}

\section{An elementary lower bound on complete bipartite subgraph density}\label{sec:probabilistic-lower-bound}

In this section, we use the probabilistic method to prove \cref{thm:probabilistic-lower-bound}. We rely on the following standard fact:

\begin{proposition}\label{prop:decrease-dim}
	In $\BR^d$, let $f$ be a $j$-flat and $h$ a hyperplane. Suppose that $h$ intersects, but is not contained in, $f$. Then $f \cap h$ is a $(j-1)$-flat.
\end{proposition}

That is, the operation of ``nontrivial intersection with a hyperplane'' reduces the dimension of a flat by one.

\begin{proof}[Proof of \cref{thm:probabilistic-lower-bound}]
	Consider the following randomized process for choosing an affine subspace $S$: Select $H_1,\ldots,H_d$ uniformly and independently from $\mathcal{H}$, and output $S := H_1 \cap \cdots \cap H_d$.
	
	Let $G$ denote the event that ``at least $\frac{\epsilon^d}2$-fraction of the points in $\CP$ lie on $S$''. We claim that $\Pr[G] \geq \frac{\epsilon^d}2$. Indeed, define the random variable $X$ as the fraction of the points in $\CP$ lying on $S$. We can write $X=\Pr_p[p\text{ lies on }S]$ where $p \sim \CP$ is uniformly random. For each $i \in [d]$, we have $\Pr_{p,H_i}[p\text{ incident to }H_i] = \epsilon$, and thus by independence, \[ \E_{H_1,\ldots,H_d}[X] = \Pr_{p,H_1,\ldots,H_d}[p\text{ incident to }S] = \Pr_{p,H_1,\ldots,H_d}\left[\bigwedge_{i=1}^d p\text{ incident to }H_i\right] = \epsilon^d. \] Conditioning, we have
	\begin{align*}
	    \epsilon^d &= \E_{H_1,\ldots,H_d}\left[X \mid X \geq \frac{\epsilon^d}2\right]\Pr_{H_1,\ldots,H_d}\left[G\right]+\E_{H_1,\ldots,H_d}\left[X \mid X < \frac{\epsilon^d}2\right]\Pr_{H_1,\ldots,H_d}\left[X < \frac{\epsilon^d}2\right] \\
	    &\leq \Pr_{H_1,\ldots,H_d}\left[G\right]+\frac{\epsilon^d}2,
	\end{align*}
	yielding the desired conclusion.
	
	For $j \in [d]$, let $B_j$ denote the event ``at most $\frac{\epsilon^d}{3d}$-fraction of the hyperplanes in $\CH$ don't contain $H_1 \cap \cdots \cap H_{j-1}$ \emph{and} $H_j$ doesn't contain $H_1 \cap \cdots \cap H_{j-1}$''. (In the case $j=1$, we define the empty intersection as all of $\BR^d$, so that $B_1$ never occurs.) For each $j \in [d]$, since $H_j$ is independent of $H_1,\ldots,H_{j-1}$, we have $\Pr[B_j] \leq \frac{\epsilon^d}{3d}$.
	 
	Hence, by the union bound, the probability of the event ``$G$ doesn't occur or $B_j$ occurs for some $j$'' is at most $1-\frac{\epsilon^d}2 + \frac{\epsilon^d}3 = 1 - \frac{\epsilon^d}6 < 1$. So by the probabilistic method, there exists a list of hyperplanes $h_1,\ldots, h_d$ such that $G$ occurs and none of the events $B_j$ occur. $G$ implies that $S = h_1 \cap \cdots \cap h_d$ contains at least $\frac{\epsilon^d}2$-fraction of the points of $\CP$. Moreover, it cannot be the case that for all $j$, $h_j$ doesn't contain $h_1 \cap \cdots \cap h_{j-1}$, since then \cref{prop:decrease-dim} implies that $S$ is either a point or empty, so $G$ cannot occur by assumption. Hence for some $j$, $\Pr_{h \samples \CH}[h \text{ contains } h_1 \cap \cdots \cap h_{j-1}] \geq \frac{\epsilon^d}{3d}$, and since $S$ is contained in $h_1 \cap \cdots \cap h_{j-1}$, we can conclude that at least $\frac{\epsilon^d}{3d}$-fraction of the hyperplanes in $\CH$ contain $S$, as desired.
\end{proof}

\section{Explicit upper bound construction with exponentially small complete bipartite subgraphs but sub-constant incidence density}\label{sec:lattice-upper-bound}

In this section, we prove \cref{thm:lattice-upper-bound} using a lattice-based explicit upper bound construction, which modifies an upper bound construction of Apfelbaum and Sharir~\cite[Theorem 1.3]{AS07} (itself based on ideas from Elekes and T\'oth~\cite{ET05}).\footnote{For ease of notation, we use $\{-1,1\}^d$ for our lattice; Apfelbaum and Sharir \cite{AS07} used $[k]$ for an arbitrary parameter $k \in \BN$.} Specifically, we will construct configurations with $n = \Theta(2^d \sqrt d)$ points, $m = \Theta(2^d)$ hyperplanes, $\I(\CP,\CH) \geq \Omega(2^{2d})$ incidences, and a complete bipartite subgraph upper bound $\rs(\CP,\CH) \leq O(2^d)$.

To prove \cref{thm:lattice-upper-bound}, we assume for simplicity that $d-1$ is a perfect square. Now consider the set of points \[ \CP := \left\{(x_1,\ldots,x_d): x_1,\ldots,x_{d-1} \in \{-1,1\}, x_d \in \left\{-2\sqrt{d-1}, \ldots, 2\sqrt{d-1} \right\} \right\} \] and the set of hyperplanes \[ \CH := \left\{\left\{x \in \BR^d : \sum_i a_i x_i = 0 \right\} : a_1,\ldots,a_{d-1} \in \{0,1\}, a_d = -1 \right\}. \] By construction, $n = |\CP| = 2^{d-1} \cdot \left(4 \sqrt {d-1} + 1\right) = \Theta(2^d \sqrt d)$ and $m = |\CH| = 2^{d-1} = \Theta(2^d)$. Also define the ``universe'' of points \[ \CU := \left\{(x_1,\ldots,x_d): x_1,\ldots,x_{d-1} \in \{-1,1\}, x_d \in \left\{-(d-1),\ldots,d-1\right\} \right\}. \] $\CU$ contains $\CP$ and has size $|\CU| = \Theta(d2^d)$.

Apfelbaum and Sharir~\cite[pp. 16-17]{AS07} proved the following three claims:

\begin{proposition}\label{claim:incidence-bound}
    $\I(\CU,\CH)=2^{2d-2}$.
\end{proposition}

\begin{proposition}\label{claim:point-bound}
    In $\BR^d$, let $f$ be a $j$-flat. Then at most $2^j$ points in $\CU$ lie on $f$.
\end{proposition}

\begin{proposition}\label{claim:hyper-bound}
    In $\BR^d$, let $f$ be a $j$-flat that is contained in some hyperplane $h \in \CH$. Then at most $2^{d-j-1}$ hyperplanes in $\CH$ contain $f$.
\end{proposition}

We include proofs of all three claims in \cref{app:AS07-lower-bound} for completeness. The latter two claims together imply that $\rs(\CU,\CH) = O(2^d)$. Finally, we prove:

\begin{proposition}\label{claim:dense-subgraph}
$\I(\CP,\CH) \geq \Omega(2^{2d})$.
\end{proposition}
\begin{proof}
We proceed probabilistically, showing that ``many'' settings of the variables $x_1,\ldots,x_{d-1}$ and $a_1,\ldots,a_{d-1}$ result in a value for $x_d = \sum_{i=1}^{d-1} a_i x_i$ that lies within the interval $[-2\sqrt{d-1},2\sqrt{d-1}]$. Consider the following experiment: Choose $x_1,\ldots,x_{d-1}$ uniformly and independently from $\{-1,1\}$, $a_1,\ldots,a_{d-1}$ uniformly and independently from $\{0,1\}$, and output \texttt{Succeed} if the sum $\sum_{i=1}^{d-1} a_i x_i$ lies in the aforementioned interval.

Each $a_i x_i$ is independently $1$ w.p. $\frac14$, $-1$ w.p. $\frac14$, and $0$ w.p. $\frac12$. Then the Chernoff-Hoeffding bound gives \[ \Pr\left[\sum_{i=1}^{d-1} a_i x_i \not\in \left[-2\sqrt {d-1}, 2\sqrt {d-1}\right] \right] \leq 2 \exp\left(-\frac{2(2\sqrt{d-1})^2}{(d-1)\cdot 2^2} \right) = \frac{2}{e^2}. \]

Thus, the experiment outputs \texttt{Succeed} with probability at least $1-\frac{2}{e^2}$. Hence using \cref{claim:incidence-bound}, $\I(\CP,\CH) \geq (1-\frac{2}{e^2}) \I(\CU,\CH) \geq \Omega(2^{2d})$.
\end{proof}

Given \cref{claim:incidence-bound,claim:point-bound,claim:hyper-bound,claim:dense-subgraph}, \cref{thm:lattice-upper-bound} follows:

\begin{proof}[Proof of \cref{thm:lattice-upper-bound}]
    Assume for simplicity that $d-1$ is a perfect square. By \cref{claim:dense-subgraph}, $\I(\CP,\CH) \geq \Omega(2^{2d})$. By \cref{claim:point-bound,claim:hyper-bound}, $\rs(\CU,\CH) \leq 2^{d-1}$, and by definition we have $\rs(\CP,\CH) \leq \rs(\CU,\CH)$. Since $n = \Theta(2^d \sqrt d)$ and $m = \Theta(2^d)$, we have $\I(\CP,\CH) \geq \Omega(nm/\sqrt d)$ and $\rs(\CP,\CH) \leq O(mn \cdot 2^{-d}/\sqrt d)$, as desired.
\end{proof}

\section{Discussion: Upper bounds from cross-intersecting families}\label{sec:set-system-upper-bound}

In this section, we report on constructions~\cite{Lov16,Lov21,Pal21,FW21} which show that $O(\sqrt d)$ is the best possible exponent we could hope for in \cref{conj:general-conj}, i.e., we exhibit explicit configurations with $\rs(\CP,\CH)/mn \leq 2^{-\Omega(\sqrt d)}$ and $\I(\CP,\CH)/mn \geq \Omega(1)$ (see \cref{thm:lattice-upper-bound-from-sets-palv} below). These constructions all arise from \emph{products of Boolean matrices}, and there are a number of natural questions in this area which we pose.

A length-$d$ Boolean vector can be viewed as the indicator of a subset of $[d]$, and using the language of set systems will provide another helpful perspective on the log-rank conjecture. (Using this perspective to construct counterexamples was suggested by~\cite{Lov21} and~\cite{FW21}.)

Let $\CA,\CB \subseteq 2^{[d]}$ be two set systems on $[d]$. Following are two notions which describe patterns among the intersection sizes $|A \cap B|$ for $A \in \CA, B \in \CB$. For $\epsilon \in [0,1]$, we say that $\CA,\CB$ are \emph{$\epsilon$-almost cross-disjoint} if $\Pr_{A \sim \CA, B \sim \CB} [A \cap B \neq \emptyset] \leq \epsilon$, and \emph{exactly cross-disjoint} in the special case $\epsilon = 0$. Following~\cite{KS05}, for $L \subseteq \{0,\ldots,d\}$, we say that $\CA,\CB$ are \emph{$L$-cross-intersecting} if for every $A \in \CA$ and $B \in \CB$, $|A \cap B| \in L$.

These notions have linear-algebraic interpretations. For $\CA,\CB \subseteq 2^{[d]}$, we can define the matrix $\Mat(\CA,\CB) \in \{0,\ldots,d\}^{\CA \times \CB}$ whose $(A,B)$-th entry is $|A \cap B|$. $\CA,\CB$ are $\epsilon$-almost cross-disjoint iff all but $\epsilon$-fraction of $\Mat(\CA,\CB)$'s entries are zeros. If $\CA,\CB$ are $L$-cross-intersecting, then $\Mat(\CA,\CB)$ is $|L|$-ary.

We have two conjectures about pairs of set systems.

\subsection{Conjecture on almost cross-disjoint set systems}

Our first conjecture would be implied by \cref{conj:general-conj}:\footnote{This implication follows from viewing each $A \subseteq \CA$ as a point in $\{0,1\}^d$ and each $B \subseteq \CB$ as a hyperplane with normal vector in $\{0,1\}^d$ and offset 0.}

\begin{conjecture}\label{conj:general-conj-for-set-systems}
The following is true for every fixed $\epsilon > 0$. Let $\CA,\CB \subseteq 2^{[d]}$ be $\epsilon$-almost cross-disjoint. Then there exist $\CR \subseteq \CA,\CS \subseteq \CB$ such that $\CR$ and $\CS$ are exactly cross-disjoint, and $|\CR||\CS| \geq |\CA| |\CB| \cdot 2^{-O(\sqrt d)}$. Equivalently, $\Mat(\CA,\CB)$ contains a $0$-monochromatic rectangle of density at least $2^{-O(\sqrt d)}$.
\end{conjecture}

The following example based on the idea of \Palvolgyi~\cite{Pal21} shows the necessity of the exponent $O(\sqrt d)$ in  \cref{conj:general-conj-for-set-systems} (and by extension \cref{conj:general-conj-for-set-systems}).

\begin{theorem}\label{thm:lattice-upper-bound-from-sets-palv}
The following is true for every $\epsilon$ in a dense subset of $(0,1)$. There exists an infinite, increasing sequence of dimensions $d_1,d_2,\ldots$, and an infinite sequence of set systems $(\CA_i,\CB_i)$ on $[d_i]$, such that $(\CA_i,\CB_i)$ is $\delta_i$-almost cross-disjoint with $\delta_i \to \epsilon$ as $i \to \infty$, but $\Mat(\CA_i,\CB_i)$ contains no 0-monochromatic submatrices of density larger than $2^{-\Omega_\epsilon(\sqrt {d_i})}$.
\end{theorem}

\begin{proof}
Consider any positive rational number $\alpha$; we will prove the theorem for $\epsilon := e^{-1/\alpha}$ (so the set of all $\epsilon$'s is dense in $(0,1)$).

Consider, in increasing order, all values $b \in \BN$ such that $a := \alpha b$ is also an integer (there are infinitely many such $b$'s by rationality). Let $d_i := ab$. We identify $[d_i]$ with $[a] \times [b]$ and subsets of $[d_i]$ with $a \times b$ Boolean matrices. Let $\CA_i = \CB_i \subset 2^{[d_i]}$ be the subset of matrices which have exactly one 1 in every column. If $n := |\CA_i|$ and $m := |\CB_i|$, then we have $n = m = a^b$.

Each $A \in \CA_i$ is disjoint from $(a-1)^b$ sets in $\CB_i$, so $(\CA_i,\CB_i)$ is $\delta_i$-almost cross-disjoint with $\delta_i = \frac{(a-1)^b}{m}$. 
Moreover, as $b$ approaches $\infty$, $\delta_i$ approaches $e^{-1/\alpha}$.

Now consider any 0-monochromatic rectangle $\CR \subset \CA_i, \CS \subset \CB_i$. Defining $R^* := \bigcup_{R \in \CR} R$ and $S^* := \bigcup_{S \in \CS} S$, we see that $R^*$ and $S^*$ must be disjoint. Hence we may assume without loss of generality that $\CR$ is the set of all matrices supported on $R^*$ and $\CS$ the set of all matrices supported on $S^*$, and that $R^*$ and $S^*$ are complementary. Defining $s_j := |\CR \cap ([a] \times \{j\})|$ (i.e., the size of $R^*$'s support in column $j$), we see that $|\CR| = s_1 \cdots s_b$ and $|\CS| = (a-s_1) \cdots (a-s_b)$. This product is maximized when each $s_j = \frac{a}2$; hence $|\CR| |\CS| \leq \left(\frac{a}2\right)^{2b}$. Thus, the complete bipartite subgraph density of $\Mat(\CA_i,\CB_i)$ is at most \[ \frac{|\CR||\CS|}{|\CA_i||\CB_i|} \leq \frac{\left(\frac{a}2\right)^{2b}}{a^{2b}} = 2^{-\Omega(b)} = 2^{-\Omega_\epsilon(\sqrt{d_i})}. \]
\end{proof}

An alternative proof was given by Lovett~\cite{Lov16,Lov21} and Fox and Wigderson~\cite{FW21}. This construction still takes $\CA=\CB$, but uses randomly sampled subsets of $[d]$ with some appropriate sparsity. It yields a similar $2^{-O(\sqrt d)}$ bound.

\subsection{Conjecture on cross-intersecting set systems}

Our second conjecture would be implied by \cref{conj:constant-sized-parallel-partitions} (and is thus equivalent to the log-rank conjecture):

\begin{conjecture}\label{conj:lrc-for-set-systems}
The following is true for every fixed $k > 0$. Let $\CA,\CB \subseteq 2^{[d]}$ be $L$-cross-intersecting, where $|L| = k$. Then there exist $\CR \subseteq \CA,\CS \subseteq \CB$, and $t \in L$, such that $\CR,\CS$ are $\{t\}$-cross-intersecting, and $|\CR||\CS| \geq |\CA||\CB|\cdot  2^{-{\polylog(d)}}$. Equivalently, $\Mat(\CA,\CB)$ contains a monochromatic rectangle of density at least $2^{-{\polylog(d)}}$.
\end{conjecture}

Lovett~\cite{Lov21} independently suggested studying the special case of \cref{conj:lrc-for-set-systems} where $k = 2$, which is perhaps the simplest combinatorial version of the log-rank conjecture. (He notes that in the subcase where $k=2$ \emph{and} $0 \in L$, \cref{conj:lrc-for-set-systems} is known to hold, since the ``log-\emph{nonnegative}-rank conjecture'' is known to hold (see, e.g., \cite[p. 57]{RY20}).) The subcase $L=\{k,k+1\}$ was described as an implication of the log-rank conjecture by Sgall~\cite{Sga99}.

$L$-cross-intersecting set systems have been studied in the extremal combinatorics literature (see e.g., \cite{FR87,Sga99,Sne03,KS05}). A typical goal in these works is to upper-bound the maximum size $|\CR||\CS|$ over all $L$-cross-intersecting set systems $(\CR,\CS)$ on $[d]$ under certain assumptions about $L$, such as being contained in a fixed number of residue classes in a fixed modulus~\cite{Sga99}. Interestingly, Frankl and R\"odl~\cite{FR87} showed that in the case $k=|L|=1$, we have $|\CR||\CS| \leq 2^d$ (see \cite[p. 556]{Sga99}). Thus, taking $\CA=\CB=2^{[d]}$, $\CA,\CB$ are $\{0,\ldots,d\}$-cross intersecting, but $\Mat(\CA,\CB)$ cannot contain any monochromatic rectangles of density greater than $2^{-d}$. This implies that in \cref{conj:lrc-for-set-systems} (and by extension \cref{conj:constant-sized-parallel-partitions}), we cannot hope to significantly improve the dependence on $k$ while maintaining $\polylog(d)$ in the exponent; in particular, \cref{conj:lrc-for-set-systems} cannot hold for $k=O(\log(|\CA||\CB|))$.

\ifnum\acmversion=1
\begin{acks}
Noah Singer was supported by the Herchel Smith Fellowship from Harvard College. Madhu Sudan was supported in part by a Simons Investigator Award and NSF Award CCF 1715187.

\else
\section*{Acknowledgements}
We would like to thank Sasha Golovnev, Raghu Meka and Santhoshini Velusamy for permission to describe their work \cite{GMSV19} here. We would also like to thank Shachar Lovett for his comments~\cite{Lov21} on the earlier version of this paper~\cite{SS21-early-version} and for his counterexample to our original Conjecture 5, as well as D\"om\"otor \Palvolgyi~\cite{Pal21} and Jacob Fox and Yuval Wigderson~\cite{FW21} for their counterexamples to the same.
We thank D\"om\"otor \Palvolgyi\ for his generous permission to build on his example in \cref{sec:set-system-upper-bound}. Finally, we acknowledge helpful comments from anonymous reviewers which significantly improved the exposition of this paper.
\fi
\ifnum\acmversion=1
\end{acks}
\fi

\appendix
\section{Proofs of claims from Apfelbaum and Sharir~\cite{AS07}}\label{app:AS07-lower-bound}

In this appendix, for completeness, we include proofs due to Apfelbaum and Sharir~\cite{AS07} which we used in the proof of \cref{thm:lattice-upper-bound}). \cref{claim:incidence-bound} has a short proof:

\begin{proof}[Proof of \cref{claim:incidence-bound}]
Consider any fixed hyperplane $h \in \CH$. Since $a_d = -1$, for any values $x_1,\ldots,x_{d-1} \in \{-1,1\}$, there is a unique value $x_d \in \{-(d-1), \ldots, d-1\}$ such that $\sum_{i=1}^d a_i x_i = 0$, i.e., $x_d = \sum_{i=1}^{d-1} a_i x_i$. Thus, $2^{d-1}$ points in $\CU$ lie on $h$. Since $|\CH|=2^{d-1}$, $\I(\CU,\CH) = 2^{2d-2}$.
\end{proof}

We need a bit more setup to prove \cref{claim:point-bound} and \cref{claim:hyper-bound}. We begin with the following helpful proposition:

\begin{proposition}\label{prop:lattice-int}
	For $\ell > 0$, let $f \subset \BR^\ell$ be a $j$-flat. Then $f$ intersects $\{-1,1\}^\ell$ in at most $2^j$ points.
\end{proposition}
\begin{proof}[Proof]
    We prove by induction on the dimension $\ell$. In the base case $\ell=1$, $f$ is the line $\BR$, and it intersects $\{-1,1\}$ in $2$ points.
    
    For general $\ell$, let $h_1$ and $h_{-1}$ be the hyperplanes defined by the equations $x_1 = 1$ and $x_1=-1$, respectively. If $f$ is contained within either hyperplane, we restrict to that hyperplane and apply the inductive hypothesis. Otherwise, by \cref{prop:decrease-dim}, $f \cap h_1$ and $f \cap h_{-1}$ are both $(j-1)$-flats or empty, so by the inductive hypothesis, $f$ intersects each of the subcubes $\{-1\} \times \{-1,1\}^{\ell-1}$ and $\{1\} \times \{-1,1\}^{\ell-1}$ in at most $2^{j-1}$ points, and hence it intersects the entire hypercube $\{-1,1\}^\ell$ in at most $2^j$ points.
\end{proof}

This lets us prove the two remaining claims.

\begin{proof}[Proof of \cref{claim:point-bound}]
	Consider any hyperplane $h \in \CH$ containing $f$. Let $h$ be defined by the equation $\sum_{i=1}^k a_i x_i = 0$, and consider the linear map \[ \phi_h : \BR^{d-1} \to h : (x_1,\ldots,x_{d-1}) \mapsto \left(x_1,\ldots,x_{d-1},\sum_{i=1}^{d-1} a_ix_i\right). \] $\phi_h$ is an isomorphism which, restricted to the hypercube $\{-1,1\}^{d-1}$, gives a bijection with the points of $\CU$ which lie on $h$.
	
	Since $f$ is contained in $h$, its preimage $\phi_h^{-1}(f)$ is a $j$-flat in $\BR^{d-1}$. By \cref{prop:lattice-int}, $\phi^{-1}_h(f)$ intersects $\{-1,1\}^{d-1}$ in at most $2^j$ points. Since $\phi_h$ restricts to a bijection between $\{-1,1\}^{d-1}$ and $\CU \cap h$, and $f$ is contained in $h$, we can conclude that $f$ intersects $\CU$ in at most $2^j$ points, as desired.
\end{proof}

\begin{proof}[Proof of \cref{claim:hyper-bound}]
	Consider any hyperplane $h \in \CH$ containing $f$. Let $h$ be defined by the equation $\langle a,x\rangle = 0$ (using inner product notation). View $f$ as the image of an affine injection $\BR^j \to \BR^d$ given by $x \mapsto Mx + v$, where $M \in \BR^{d \times j}$ has full rank and $v \in \BR^d$.
	
	Since $h$ contains $f$, for any $y \in \BR^j$, $\langle a, My + v\rangle = 0$. Hence $\langle a, v\rangle = 0$ (plugging in $y = 0$), so $\langle a, My\rangle = 0$ for all $y \in \BR^j$ (subtracting). Now $a^\top M$ is simply a vector in $\BR^j$; if its inner product with all $y \in \BR^j$ is zero, then it is zero. Thus, $a \in \ker(M^\top)$.
	
	Let $K := \operatorname{ker}(M^\top)$. By rank-nullity, and since row-rank equals column-rank, \[ \dim(K) = d - \dim(\operatorname{im} (M^\top)) = d - \dim(\operatorname{im}(M)) = d - j. \]
	
	Now consider the hyperplane $h'$ defined by the equation $x_d = -1$. We have $a \in K \cap h' \cap \{-1,1\}^d$. But $K$ is not contained in $h'$, since $K$ contains the origin while $h'$ does not. Hence $K \cap h'$ is a $(d-j-1)$-flat by \cref{prop:decrease-dim}, so by \cref{prop:lattice-int}, it can intersect $\{-1,1\}^d$ in at most $2^{d-j-1}$ points, upper-bounding the number of possible $a$'s.
\end{proof}

\ifnum\acmversion=1
\bibliographystyle{ACM-Reference-Format}
\bibliography{logrank}
\else
\printbibliography
\fi

\end{document}

%% file: header.tex
\usepackage[margin=1in]{geometry}
\usepackage[utf8]{inputenc}
\usepackage{bm}
\usepackage{graphicx}
\usepackage{color,xcolor}
\usepackage{amsmath,amsfonts, amssymb}

\usepackage{amsthm,thmtools}

\PassOptionsToPackage{hyphens}{url}
\usepackage[colorlinks=true, allcolors=blue]{hyperref}
\usepackage[capitalise,nameinlink]{cleveref}

\usepackage[style=alphabetic, backend=biber, minalphanames=3, maxalphanames=4, maxbibnames=99]{biblatex}

\crefformat{section}{#2\S#1#3}
\crefformat{subsection}{#2\S#1#3}
\crefformat{subsubsection}{#2\S#1#3}

\declaretheoremstyle[bodyfont=\it,qed=\qedsymbol]{noproofstyle}

\numberwithin{equation}{section}

\declaretheorem[name=Observation,numbered=no]{observation*}

\declaretheorem[numberlike=equation]{theorem}

\declaretheorem[name=Theorem,numbered=no]{theorem*}

\declaretheorem[numberlike=equation]{lemma}
\declaretheorem[name=Lemma,numbered=no]{lemma*}

\declaretheorem[name=Corollary,numbered=no]{corollary*}

\declaretheorem[numberlike=equation]{proposition}
\declaretheorem[name=Proposition,numbered=no]{proposition*}

\declaretheorem[name=Claim,numbered=no]{claim*}

\declaretheorem[numberlike=equation]{conjecture}
\declaretheorem[name=Conjecture,numbered=no]{conjecture*}

\declaretheorem[name=Question,numbered=no]{question*}

\declaretheoremstyle[bodyfont=\it]{defstyle} 

\declaretheorem[numberlike=equation,style=defstyle]{definition}
\declaretheorem[unnumbered,name=Definition,style=defstyle]{definition*}

\declaretheorem[unnumbered,name=Example,style=defstyle]{example*}

\declaretheorem[unnumbered,name=Notation=defstyle]{notation*}

\declaretheorem[unnumbered,name=Construction,style=defstyle]{construction*}

\declaretheoremstyle[]{rmkstyle}

%% file: macros.tex
\newcommand{\BN}{\mathbb N}
\newcommand{\BR}{\mathbb R}
\newcommand{\CA}{\mathcal A}
\newcommand{\CB}{\mathcal B}
\newcommand{\CH}{\mathcal H}

\newcommand{\CP}{\mathcal P}
\newcommand{\CR}{\mathcal R}
\newcommand{\CS}{\mathcal S}
\newcommand{\CU}{\mathcal U}
\newcommand{\CX}{\mathcal X}
\newcommand{\CY}{\mathcal Y}
\renewcommand{\tilde}{\widetilde}

\newcommand{\Gr}{\operatorname{G}}
\newcommand{\I}{\operatorname{I}}
\newcommand{\rs}{\operatorname{rs}}
\newcommand{\rank}{\operatorname{rank}}
\newcommand{\CC}{\operatorname{CC}}
\newcommand{\polylog}{\operatorname{polylog}}

\newcommand{\Mat}{\mathsf{Mat}}
\newcommand{\Con}{\mathsf{Con}}

\newcommand{\Palvolgyi}{P\'alv\"olgyi}

\newcommand{\samples}{\sim}

\newcommand{\E}{\mathop{\mathbb{E}}}